\documentclass[12pt,amstex]{amsart}
\usepackage{stmaryrd}

\usepackage{epsfig}
\usepackage{amsmath}
\usepackage{amssymb}
\usepackage{amscd}
\usepackage{graphicx}

\usepackage{pstricks}

\usepackage{tocvsec2}

\usepackage{hyperref}

\input xy
\xyoption{all}

\newtheorem{thm}{Theorem}[section]
\newtheorem{lem}[thm]{Lemma}

\newtheorem{prop}[thm]{Proposition}
\newtheorem{ques}[thm]{Question}
\newtheorem{cor}[thm]{Corollary}
\newtheorem{conj}[thm]{Cojecture}
\theoremstyle{definition}
\newtheorem{de}[thm]{Definition}
\newtheorem{exam}[thm]{Example}
\theoremstyle{remark}
\newtheorem{rem}[thm]{Remark}

\numberwithin{equation}{section}

\def \N {\mathbb N}

\def \Z {\mathbb Z}

\def \E {\mathbb E}

\def \O {\mathcal{O}}

\def \F {\mathcal F}
\def \G {\mathcal{G}}

\def \X {\mathcal{X}}
\def \Y {\mathcal{Y}}

\def \Q {{\bf Q}}
\def \RP {{\bf RP}}

\def \id {{\rm id}}

\def \a {\alpha }

\def \ep {\epsilon}
\def \d {\delta}
\def \D {\Delta}

\def \T {\mathbb T}

\begin{document}
\title{Topological characteristic factors along cubes of minimal systems}

\author{Fangzhou Cai}
\author{Song Shao}

\address{Wu Wen-Tsun Key Laboratory of Mathematics, USTC, Chinese Academy of Sciences and
Department of Mathematics, University of Science and Technology of China,
Hefei, Anhui, 230026, P.R. China.}

\email{cfz@mail.ustc.edu.cn}
\email{songshao@ustc.edu.cn}

%\keywords{regionally proximal relation; minimal flow}

\thanks{This research is supported by NNSF of China (11571335, 11431012 ) and by ¡°the Fundamental Research Funds for the Central Universities¡±.}

\date{}

\begin{abstract}
In this paper we study the topological characteristic factors along cubes of minimal systems. It is shown that up to proximal extensions the pro-nilfactors are the topological characteristic factors along cubes of minimal systems. In particular, for a distal minimal system, the maximal $(d-1)$-step pro-nilfactor is the topological cubic
characteristic factor of order $d$.
\end{abstract}

\maketitle

%\markboth{ergodic}{S. Shao and X.D. Ye}

%\newpage

%\tableofcontents \settocdepth{subsection}

%\newpage

\section{Introduction}

This paper is motivated by the work of Glasner on topological characteristic factors in topological dynamics \cite{G94} and the work of Host and Kra on the multiple ergodic averages \cite{HK05}. In \cite{G94}, Glasner studied the topological characteristic factors along arithmetic progressions, and his work is the counterpart of Furstenberg's work \cite{F77} in topological dynamics. The present work is dedicated to the topological characteristic factors along cubes, which may be considered as the counterpart of \cite{HK05} in topological dynamics.

\subsection{Characteristic factors in ergodic theory}\
\medskip

The connection between ergodic theory and additive combinatorics was
built in the 1970's with Furstenberg's beautiful proof of
Szemer\'edi's theorem via ergodic theory \cite{F77}. Furstenberg \cite{F77} proved
Szemer\'edi's theorem via the following multiple recurrence theorem: let $T$ be
a measure-preserving transformation on the probability space
$(X,\X,\mu)$, then for every integer $d \ge 1$ and $A\in \mathcal{X}$ with positive
measure,
\begin{equation*}
    \liminf_{N\to \infty} \frac{1}{N}\sum_{n=0}^{N-1}
    \mu(A\cap T^{-n}A\cap T^{-2n}A\cap \ldots \cap T^{-dn}A)>0.
\end{equation*}
So it is natural to ask about the convergence of these averages, or
more generally about the convergence in $L^2(X,\mu)$ of the {\em multiple
ergodic averages} (or called {\em non-conventional averages})
$$ \frac 1 N\sum_{n=0}^{N-1}f_1(T^nx)\ldots
f_d(T^{dn}x) ,$$ where $f_1, \ldots , f_d \in L^\infty(X,\mu)$.
After nearly 30 years' efforts of many researchers, this problem was
finally solved in \cite{HK05} (see \cite{Z} for an another proof).

In the study of multiple ergodic averages, the idea of characteristic factors play an very important role. This idea was suggested by Furstenberg in \cite{F77}, and the notion of ``characteristic factors'' was
first introduced in a paper by Furstenberg and Weiss \cite{FW96}.

\begin{de}\cite{FW96}\label{de-charac-mps}
Let $(X,\X,\mu, T)$ be a measurable system and $(Y,\Y,\mu, T)$ be a factor of $X$. Let $\{p_1,\ldots, p_d\}$ be a family of integer valued polynomials, $d\in \N$. We say that $Y$ is a {\em $L^2$(resp. a.e.)-characteristic factor} of $X$ for the scheme $\{p_1,\ldots, p_d\}$ if for all $f_1,\ldots,f_d\in L^\infty(X,\X,\mu)$,
\begin{equation*}
\begin{split}
  \frac{1}{N}\sum_{n=0}^{N-1} & f_1(T^{p_1(n)}x) f_2(T^{p_2(n)}x)\ldots f_d(T^{p_d(n)}x)\\
  -  &\frac{1}{N}\sum_{n=0}^{N-1} \E(f_1|\Y)(T^{p_1(n)}x) \E(f_2|\Y)(T^{p_2(n)}x)\ldots \E(f_d|\Y)(T^{p_d(n)}x)\to 0
\end{split}
\end{equation*}
in $L^2(X,\X,\mu)$ (resp. almost everywhere).
\end{de}

Finding a characteristic factor for
a scheme often gives a reduction of the problem of evaluating limit behavior
of multiple ergodic averages to special systems. The
structure theorem of \cite{HK05, Z} states that for an ergodic system $(X,\X,\mu,T)$ if one wants to
understand the multiple ergodic averages
$$ \frac 1 N\sum_{n=0}^{N-1}f_1(T^nx)\ldots f_d(T^{dn}x) ,$$
one can replace each function $f_i$ by its conditional expectation
on some $d-1$-step pro-nilsystem ($0$-step system is a trivial system and $1$-step pro-nilsystem is the Kroneker's
one). Thus one can reduce the problem to the study of the same
average in a nilsystem, i.e. reducing the average in an arbitrary
system to a more tractable question.
%Recall that A {\em $d$-step nilsystem} is a pair $(X, T)$ where $X$ is a compact homogeneous space of a $d$-step nilpotent group $G$ and $T$ is a translation of $X$ defined by an element of $G$. When $G$ is a nilpotent Lie group, $X$ is called a {\em $d$-step nilmanifold}, and if $G$ is an inverse limit of nilpotent Lie groups, $(X,T)$ is called a {\em $d$-step pro-nilsystem} or a {\em system of order $d$}.

In \cite{HK05}, lots of useful tools, such as dynamical
parallelepipeds, ergodic uniformity seminorms etc., were introduced
in the study of dynamical systems. One of main results of \cite{HK05} is the following theorem of multiple ergodic averages along cubes.

\begin{thm}\cite[Theorem 1.2]{HK05}\label{HK-cube}
Let $(X,\X, \mu, T)$ be an measure preserving probability system, and $d\in \N$. Then
for functions $f_\ep\in L^{\infty}(\mu), \
\ep\in \{0,1\}^d, \ep\not=(0,\ldots,0) $, the averages
\begin{equation}\label{C1}
   \prod_{i=1}^d \frac{1}{N_i-M_i} \cdot \sum_{{\bf n}\in [M_1,N_1)\times \ldots\times [M_d,N_d)} \prod _{(0,\ldots,0)\neq\ep\in \{0,1\}^d}
    f_\ep (T^{{\bf n}\cdot \ep}x)
\end{equation}
converge in $L^2(X)$ as $N_1-M_1, N_2-M_2,\ldots,N_d-M_d$ tend to $+\infty$.
\end{thm}

One may define the characteristic factor of (\ref{C1}) as defined in Definition \ref{de-charac-mps}.
To prove theorem above the authors in \cite{HK05} showed that the $d$-dimensional average along cubes has the same characteristic factor as the average along arithmetic progressions of length $d-1$, which is a $d-1$-step pro-nilsystem. The main result of the paper is to give the topological counterpart of this fact, that is, to show that pro-nilfactors are the topological characteristic factors along cubes of minimal systems.

\subsection{Topological characteristic factors along arithmetic progressions}\
\medskip

The counterpart of characteristic factors in topological dynamics was first studied by
Glasner in \cite{G94}. In \cite{G94} Glasner studied the characteristic
factors for the transformation $T\times T^2\times \ldots \times T^d$.

\begin{de}
Let $(X,T)$ be a topological system and let $\pi : (X,T)\rightarrow (Y,T)$ be a
factor map. A subset $L$ of $X$ is called {\em $\pi$-saturated} if
$\{x\in L: \pi^{-1}(\pi(x))\subseteq L\}=L$, i.e. $L=\pi^{-1}(\pi(L))$.
%$x\in L$ iff $x'\in L$ whenever $\pi(x')=\pi(x)$.
\end{de}

Here is the definition of topological characteristic factors along arithmetic progressions:

\begin{de}\cite{G94}
Let $(X,T)$ be a system and $d\in \N$. Let $\pi: (X,T)\rightarrow (Y,T)$ be a factor map and $\sigma_d=T\times T^2\times
\ldots\times T^d$. $(Y,T)$ is said to be a {\em topological
characteristic factor (along arithmetic progressions) of order $d$ } if there exists a dense $G_\d$
set $X_0$ of $X$ such that for each $x\in X_0$ the orbit
closure $L_x=\overline{\O}(\underbrace{(x, \ldots,x)}_{2^d \
\text{times}}, \sigma_d)$ is $\underbrace{\pi\times \ldots \times
\pi}_{d \ \text{times}}$ saturated. That is, $(x_1,x_2,\ldots, x_d)\in L_x$
iff $(x_1',x_2',\ldots, x_d')\in L_x$ whenever for all $i$,
$\pi(x_i)=\pi(x_i')$.
\end{de}

In \cite{G94} it is said that this notation was suggested by Furstenberg and it is systematically studied.
In \cite{G94}, it is shown that up to a canonically defined proximal extension, a characteristic family for $T\times T^2\times \ldots \times T^d$ is the family of canonical PI flows of class $d-1$. In particular,
if $(X,T)$ is a distal minimal system, then its largest class $d-1$ distal factor is its topological characteristic factor of order $d$; if $(X,T)$ is a weakly mixing system $(X,T)$, then the trivial system is its topological characteristic factor. For more related results and details please refer to \cite{G94}.

\medskip

A unsolved problem is:

\begin{conj}\label{character}
If $(X,T)$ is a distal minimal system, then its maximal $(d-1)$-step pro-nilfactor
is its topological characteristic factor along arithmetic progressions of order $d$.
\end{conj}

%In \cite{G94}, it is shown that
%\begin{thm}\cite[Theorem 2.3]{G94}\label{Eli94}
%\begin{enumerate}
%  \item If $(X,T)$ is a distal minimal system, then its largest class $d$ distal factor is its topological characteristic factor of order $d+1$;
%\item If $(X,T)$ is a weakly mixing system $(X,T)$, then all its topological characteristic factors are trivial systems.
%\end{enumerate}
%\end{thm}

\subsection{Topological characteristic factors along cubes and main results of the paper}\
\medskip

First we define topological characteristic factors along cubes.
The transformation group related to (\ref{C1}) is the face group $\F^{[d]}$. Please refer to next section for precise definition. Note that the group $\F^{[d]}$ acts on $X^{2^d}$ and it acts on the first coordinate as an identity map.

%The key role is the face group $\F^{[d]}(X)$ of dimension $d$ and the cube group $\G^{[d]}(X)$ of dimension $d$. See definitions in the next section. Note that groups $\F^{[d]}$ and $\G^{[d]}$ act on the product space $X^{2^d}$ and the group $\F^{[d]}$ act on the first coordinate as an identity map.

\begin{de}\label{de-characristic-factor}
Let $(X,T)$ be a system and $d\in \N$. Let $\pi: (X,T)\rightarrow (Y,T)$ be a factor map. The system $(Y,T)$ is said to be a {\em topological cubic
characteristic factor of order $d$ } or {\em topological characteristic factor along cubes of order $d$} if there exists a dense $G_\d$
set $X_0$ of $X$ such that for each $x\in X_0$ the set $F_x=p_* \left(\overline{\O}(\underbrace{(x, \ldots,x)}_{2^d \
\text{times}}, \F^{[d]})\right)$ is $\pi^{2^d-1}=\underbrace{\pi\times \ldots \times
\pi}_{2^d-1 \ \text{times}}$ saturated, where $p_*: X^{2^d}\rightarrow X^{2^d-1}$ is the projection on the later $2^d-1$ coordinates. That is, for each $x\in X_0$,
$$\overline{\O}(\underbrace{(x, \ldots,x)}_{2^d \
\text{times}}, \F^{[d]})=\{x\}\times (\pi^{2^d-1})^{-1}(\pi^{2^d-1}F_x).$$
\end{de}

One of the main results of the paper is that up to proximal extensions the maximal $(d-1)$-step pro-nilfactor is the topological cubic
characteristic factor of order $d$. To be precise, we will show the following theorem:

\begin{thm}\label{thm-main}
Let $(X,T)$ be a minimal system and $d\in \N$. Let $\pi:(X,T)\rightarrow (Z_{d-1},T)$ be the factor map to the maximal $(d-1)$-step pro-nilfactor. Then there is a commutative diagram of homomorphisms of minimal flows
\begin{equation*}
  \xymatrix{
    X\ar[d]_{\pi}& X^\prime\ar[d]^{\pi^\prime}\ar[l]_{\theta^\prime} \\
    Z_{d-1}  &Y^\prime\ar[l]_{\theta}
    }
\end{equation*}
such that $(Y',T)$ is the topological cubic
characteristic factor of order $d$ of $(X',T)$,
where $\theta,\theta^\prime$ are proximal extensions.
\end{thm}

When $X$ is distal or weakly mixing, the proximal extensions $\theta, \theta'$ in theorem are trivial (i.e. isomorphisms). That is:

\begin{cor}\

\medskip

\begin{enumerate}
  \item Let $(X,T)$ be a minimal distal system and $d\in \N$. Then the maximal $(d-1)$-step pro-nilfactor is the topological cubic characteristic factor of order $d$.
  \item Let $(X,T)$ be a minimal weakly mixing system and $d\in \N$. Then the trivial system is the  topological cubic characteristic factor of order $d$.
\end{enumerate}
\end{cor}

%\subsection{Questions}

We do not know whether one may remove proximal extensions in Theorem \ref{thm-main}, that is, we have the following question:

\begin{ques}\label{character}
Let $(X,T)$ be a minimal system and $d\in \N$. Is the maximal $(d-1)$-step pro-nilfactor the topological cubic characteristic factor of order $d$?
\end{ques}

%\subsection*{Organization of the paper}

%We organize the paper as follows. In Section

\medskip
\noindent{\bf Acknowledgments:} We would like to thank Professor Wen Huang and Professor Xiangdong Ye for very useful suggestions.

\section{Preliminaries}

In the article, integers, nonnegative integers and natural numbers
are denoted by $\Z$, $\Z_+$ and $\N$ respectively. In the
following subsections we give the basic background in topological
dynamics necessary for the article.
%More details can be found in \cite{Au88, G, Vr}.

\subsection{Topological dynamical systems}

By a {\it topological dynamical system} (TDS for short) we mean a
pair $(X,T)$ where $X$ is a compact metric space (with metric $d$)
and $T:X \to X$ is a homeomorphism.  For $n \geq 2$ one
writes $(X^n,T^{(n)})$ for the $n$-fold product system $(X\times
\cdots \times X,T\times \cdots \times T)$. The diagonal of $X^n$
is
$$\Delta_n(X)=\{(x,\ldots,x)\in X^n: x\in X\}.$$
%and $$\Delta^{(n)}(X)=\{(x_1,\ldots,x_n) \in X^n: \text{ for some
%$i\neq j$ }, x_i=x_j  \}.$$
When $n=2$ one writes
$\Delta_2(X)=\Delta(X)$.
The {\it orbit of $x \in X$}
is given by $\O(x,T)=\{T^nx: n\in \Z\}$.
For convenience, sometimes one denotes the orbit closure of $x\in X$
under $T$ by $\overline{\O}(x,T)$ or $\overline{\O}(x)$, instead of
$\overline{\O(x, T)}$.

\medskip

A TDS $(X,T)$ is {\it transitive} if for any two nonempty open
sets $U$ and $V$ there is $n\in \Z$ such that $U\cap T^{-n}V\neq
\emptyset$. It is {\it point transitive} if there exists $x\in X$
such that $\overline{\O(x,T)}=X$; such $x$ is called a {\it
transitive point}. One says $(X,T)$ is {\it weakly mixing} if the product system
$(X^2,T^{(2)})$ is transitive. A TDS $(X,T)$ is {\it minimal} if
$\overline{\O(x,T)}=X$ for every $x\in X$. A point $x \in X $ is
{\it minimal} or {\it almost periodic} if the subsystem
$(\overline{\O(x,T)},T)$ is minimal.
%If $(X,T)$ is minimal then $(X^n, T^{(n)})$ has dense minimal points.

\medskip

A {\em factor map} $\pi: X\rightarrow Y$ between the TDS $(X,T)$
and $(Y,S)$ is a continuous onto map which intertwines the
actions; one says that $(Y,S)$ is a {\it factor} of $(X,T)$ and
that $(X,T)$ is an {\it extension} of $(Y,S)$.
%The systems are said to be {\it conjugate} if $\pi$ is bijective.

\medskip

Generally, a {\em topological dynamical systems} is a triple
$\X=(X, G, \Pi)$, where $X$ is a compact $T_2$ space, $G$ is a
$T_2$ topological group and $\Pi: G\times X\rightarrow X$ is a
continuous map such that $\Pi(e,x)=x$ and
$\Pi(s,\Pi(t,x))=\Pi(st,x)$. We should fix $G$ and suppress the
action symbol. In lots of literatures, $\X$ is also called a {\em
topological transformation group} or a {\em flow}. An
analogous definition can be given if  $G$ is a semigroup. Also, the notions of
transitivity, minimality and weak mixing are naturally generalized to
group actions.
%To be simple, we always assume that $G$ has a discrete topology unless we point out individually in this paper.

\subsection{Cubes and faces}\
\medskip

Cube groups and face groups are introduced by Host and Kra in dynamical systems. Please refer to \cite{HK05, HK18} for more details.

\medskip

Let $X$ be a set, let $d\ge 1$ be an integer, and write
$[d] = \{1, 2,\ldots , d\}$. We view element in $\{0, 1\}^d$ as a sequence $\ep=\ep_1\ldots \ep_d$ of $0'$s and
$1'$s.

Let $V_d=\{0,1\}^d$ and $V_d^*=V_d\setminus \{{\bf
0}\}$.
If ${\bf n} = (n_1,\ldots, n_d)\in \Z^d$ and $\ep\in \{0,1\}^d$, we
define
$${\bf n}\cdot \ep = \sum_{i=1}^d n_i\ep_i .$$

We denote $X^{2^d}$ by $X^{[d]}$. A point ${\bf x}\in X^{[d]}$ can
be written as
$${\bf x} = (x_\ep :\ep\in \{0,1\}^d ). $$
Hence $x_{\bf 0}$ is the first coordinate of ${\bf x}$.
As examples, points in $X^{[2]}$ are like
$$(x_{00},x_{10},x_{01},x_{11}).$$

For $x \in X$, we write $x^{[d]} = (x, x,\ldots , x)\in  X^{[d]}$.
The diagonal of $X^{[d]}$ is $\D^{[d]} = \{x^{[d]}: x\in X\}$.
Usually, when $d=1$, denote diagonal by $\D_X$ or $\D$ instead of
$\D^{[1]}$.

A point ${\bf x} \in X^{[d]}$ can be decomposed as ${\bf x} = ({\bf
x'},{\bf  x''})$ with ${\bf x}', {\bf x}''\in X^{[d-1]}$, where
${\bf x}' = (x_{\ep0} : \ep\in \{0,1\}^{d-1})$ and ${\bf x}''=
(x_{\ep1} : \ep\in \{0,1\}^{d-1})$. We can also isolate the first
coordinate, writing $X^{[d]}_* = X^{2^d-1}$ and then writing a point
${\bf x}\in X^{[d]}$ as ${\bf x} = (x_{\bf 0}, {\bf x}_*)$, where
${\bf x}_*= (x_\ep : \ep\neq {\bf 0}) \in X^{[d]}_*$.

%\subsubsection{}

%The {\em faces} of dimension $r$ of a point in ${\bf x}\in X^{[d]}$ are defined as follows. Let $J\subset [d]$ with $|J| = d-r$ and $\xi\in \{0,1\}^{d-r}$. The elements $(x_\ep : \ep\in \{0,1\}^{d}, \ep_J=\xi)$ of $X^{[r]}$ are called {\em faces of dimension $r$ } of ${\bf x}$, where $\ep_J = (\ep_i : i \in J)$. Thus any face of dimension $r$ defines a natural projection from $X^{[d]}$ to $X^{[r]}$, and we call this the projection along this face.

\subsection{Dynamical parallelepipeds}

\begin{de}\label{de-Q}
Let $(X, T)$ be a topological dynamical system and let $d\ge 1$ be
an integer. We define $\Q^{[d]}(X)$ to be the closure in $X^{[d]}$
of elements of the form $$(T^{{\bf n}\cdot \ep}x=T^{n_1\ep_1+\ldots
+ n_d\ep_d}x: \ep= (\ep_1,\ldots,\ep_d)\in\{0,1\}^d) ,$$ where ${\bf
n} = (n_1,\ldots , n_d)\in \Z^d$ and $ x\in X$. When there is no
ambiguity, we write $\Q^{[d]}$ instead of $\Q^{[d]}(X)$. An element
of $\Q^{[d]}(X)$ is called a (dynamical) {\em parallelepiped of
dimension $d$}.
\end{de}

As examples, $\Q^{[2]}$ is the closure in $X^{[2]}=X^4$ of the set
$$\{(x, T^mx, T^nx, T^{n+m}x) : x \in X, m, n \in \Z\}$$ and $\Q^{[3]}$
is the closure in $X^{[3]}=X^8$ of the set $$\{(x, T^mx, T^nx,
T^{m+n}x, T^px, T^{m+p}x, T^{n+p}x, T^{m+n+p}x) : x\in X, m, n, p\in
\Z\}.$$

\begin{de}
Let $\phi: X\rightarrow Y$ and $d\in \N$. Define $\phi^{[d]}:
X^{[d]}\rightarrow Y^{[d]}$ by $(\phi^{[d]}{\bf x})_\ep=\phi x_\ep$
for every ${\bf x}\in X^{[d]}$ and every $\ep\subset [d]$.
Let $(X, T)$ be a system and $d\ge 1$ be an integer. The {\em
diagonal transformation} of $X^{[d]}$ is the map $T^{[d]}$.
\end{de}

\begin{de}
{\em Face transformations} are defined inductively as follows: Let
$T^{[0]}=T$, $T^{[1]}_1=\id \times T$. If
$\{T^{[d-1]}_j\}_{j=1}^{d-1}$ is defined already, then set
\begin{equation}\label{def-T[d]}
\begin{split}
T^{[d]}_j&=T^{[d-1]}_j\times T^{[d-1]}_j, \ j\in \{1,2,\ldots, d-1\},\\
T^{[d]}_d&=\id ^{[d-1]}\times T^{[d-1]}.
\end{split}
\end{equation}
\end{de}

%It is easy to see that for $j\in [d]$, the face transformation
%$T^{[d]}_j : X^{[d]}\rightarrow X^{[d]}$ can be defined by, for
%every ${\bf x} \in X^{[d]}$ and $\ep\subset [d]$,$$T^{[d]}_j{\bf x}=
%\left\{
 % \begin{array}{ll}
  %  (T^{[d]}_j{\bf x})_\ep=Tx_\ep, & \hbox{$j\in \ep$;} \\
  %  (T^{[d]}_j{\bf x})_\ep=x_\ep, & \hbox{$j\not \in \ep$.}
  %\end{array}
%\right.$$

The {\em face group} of dimension $d$ is the group $\F^{[d]}(X)$ of
transformations of $X^{[d]}$ spanned by the face transformations.
The {\em cube group} or {\em parallelepiped group} of dimension $d$ is the group
$\G^{[d]}(X)$ spanned by the diagonal transformation and the face
transformations. We often write $\F^{[d]}$ and $\G^{[d]}$ instead of
$\F^{[d]}(X)$ and $\G^{[d]}(X)$, respectively. For $\G^{[d]}$ and
$\F^{[d]}$, we use similar notations to that used for $X^{[d]}$:
namely, an element of either of these groups is written as $S =
(S_\ep : \ep\in\{0,1\}^d)$. In particular, $\F^{[d]} =\{S\in
\G^{[d]}: S_{\bf 0} ={\rm id}\}$.
Let $p_*: X^{[d]}\rightarrow X^{[d]}_*$ be the projection. Then all transformations of $\G^{[d]}$ and $\F^{[d]}$ factor through the projection $p_*$ and induce transformations of $X^{[d]}_*$. We denote the corresponding groups by $\G^{[d]}_*$ and $\F^{[d]}_*$ respectively.

\medskip

For convenience, we denote the orbit closure of ${\bf x}\in X^{[d]}$
under $\F^{[d]}$ by $\overline{\F^{[d]}}({\bf x})$, instead of
$\overline{\O({\bf x}, \F^{[d]})}$.
It is easy to verify that $\Q^{[d]}$ is the closure in $X^{[d]}$ of
$$\{Sx^{[d]} : S\in \F^{[d]}, x\in X\}.$$
If $x$ is a transitive point of $X$, then $\Q^{[d]}$ is the closed
orbit of $x^{[d]}$ under the group $\G^{[d]}$.

If $(X,T)$ is minimal, then for all $x\in X$, $(\overline{\F^{[d]}}(x^{[d]}), \F^{[d]})$ is minimal, and $(\Q^{[d]},\G^{[d]})$ is minimal \cite{SY}.

\subsection{Nilmanifolds and nilsystems}\
\medskip

Let $G$ be a group. For $g, h\in G$ and $A,B \subset G$, we write $[g, h] =
ghg^{-1}h^{-1}$ for the commutator of $g$ and $h$ and
$[A,B]$ for the subgroup spanned by $\{[a, b] : a \in A, b\in B\}$.
The commutator subgroups $G_j$, $j\ge 1$, are defined inductively by
setting $G_1 = G$ and $G_{j+1} = [G_j ,G]$. Let $d \ge 1$ be an
integer. We say that $G$ is {\em $d$-step nilpotent} if $G_{d+1}$ is
the trivial subgroup.

\medskip

Let $G$ be a $d$-step nilpotent Lie group and $\Gamma$ be a discrete
cocompact subgroup of $G$. The compact manifold $X = G/\Gamma$ is
called a {\em $d$-step nilmanifold}. The group $G$ acts on $X$ by
left translations and we write this action as $(g, x)\mapsto gx$.
The Haar measure $\mu$ of $X$ is the unique probability measure on
$X$ invariant under this action. Let $\tau\in G$ and $T$ be the
transformation $x\mapsto \tau x$ of $X$. Then $(X, \mu, T)$ is
called a {\em $d$-step nilsystem}. In the topological setting we omit the measure
and just say that $(X,T)$ is a $d$-step nilsystem.

\medskip

We will need to use inverse limits of nilsystems, so we recall the
definition of a sequential inverse limit of systems. If
$(X_i,T_i)_{i\in \N}$ are systems with $diam(X_i)\le 1$ and
$\pi_i: X_{i+1}\rightarrow X_i$ are factor maps, the {\em inverse
limit} of the systems is defined to be the compact subset of
$\prod_{i\in \N}X_i$ given by $\{ (x_i)_{i\in \N }: \pi_i(x_{i+1}) =
x_i\}$, and we denote it by
$\lim\limits_{\longleftarrow}(X_i,T_i)_{i\in\N}$. It is a compact
metric space endowed with the distance $\rho((x_{i})_{i\in\N}, (y_{i})_{i\in
\N}) = \sum_{i\in \N} 1/2^i \rho_i(x_i, y_i )$, where $\rho_{i}$ is the metric in
$X_{i}$. We note that the
maps $T_i$ induce naturally a transformation $T$ on the inverse
limit.

\medskip

The following structure theorem characterizes inverse limits of
nilsystems using dynamical parallelepipeds.

\begin{thm}[Host-Kra-Maass]\cite[Theorem 1.2.]{HKM}\label{HKM}
Assume that $(X, T)$ is a transitive topological dynamical system
and let $d \ge 2$ be an integer. The following properties are
equivalent:
\begin{enumerate}
  \item If ${\bf x}, {\bf y} \in \Q^{[d]}$ have $2^d-1$ coordinates in common, then ${\bf x} = {\bf y}$.
  \item If $x, y \in X$ are such that $(x, y,\ldots , y) \in  \Q^{[d]}$,
then $x = y$.
  \item $X$ is an inverse limit of $(d-1)$-step minimal
nilsystems.
\end{enumerate}
\end{thm}

A transitive system satisfying one of the equivalent properties
above is called  a {\em system of order $(d-1)$} or a {\em $(d-1)$-step pro-nilsystem}.

\subsection{Proximal, distal and regionally proximal relations}\
\medskip

Let $(X,T)$ be a TDS. Fix $(x,y)\in X^2$. It is a {\it proximal}
pair if $\inf\limits_{n\in \Z} d(T^nx, T^ny)=0$; it is a {\it
distal} pair if it is not proximal. Denote by ${\bf P}(X,T)$ the set of proximal pairs of $(X,T)$. It is also called the proximal relation. A TDS $(X,T)$ is {\it equicontinuous} if for every $\ep>0$ there
exists $\d>0$ such that $d(x,y)< \d$ implies $d(T^nx,T^ny)<\ep$
for every $n\in \Z$. It is {\it distal} if ${\bf P}(X,T)= \D(X)$. Any equicontinuous system is distal.

\medskip

Let $(X,T)$ be a minimal system. The regionally
proximal relation $\RP(X,T)$ is defined as: $(x,y)\in \RP$ if
there are sequences $x_i,y_i\in X, n_i\in \Z$ such that $x_i\to x,
y_i\to y$ and $(T\times T)^{n_i}(x_i,y_i)\to (z,z)$, $i\to \infty$,
for some $z\in X$. It is well known that $\RP(X,T)$ is
an invariant closed equivalence relation and this relation
defines the {\em maximal equicontinuous factor} $X_{eq}=X/\RP(X,T)$ of
$(X,T)$ (for example see \cite{Vr}).% (see \cite[Chapter 9]{Au88}).

%If $P(X,T)=Q(X,T)$ then one says it is a {\em proximal equicontinuous} system. It can be checked that $(X,T)$ is proximally equicontinuous if and only if it is a proximal extension of an equicontinuous system.

\subsection{Regionally proximal relation
of order $d$}

\begin{de}
Let $(X, T)$ be a system and let $d\in \N$. The points $x, y \in X$ are
said to be {\em regionally proximal of order $d$} if for any $\d  >
0$, there exist $x', y'\in X$ and a vector ${\bf n} = (n_1,\ldots ,
n_d)\in\Z^d$ such that $\rho (x, x') < \d, \rho (y, y') <\d$, and $$
\rho (T^{{\bf n}\cdot \ep}x', T^{{\bf n}\cdot \ep}y') < \d\
\text{for any $\ep\in \{0,1\}^d\setminus \{{\bf 0}\}$}.$$ In other words, there
exists $S\in \F^{[d]}$ such that $\rho (S_\ep x', S_\ep y') <\d$ for
every $\ep\in \{0,1\}^d\setminus \{{\bf 0}\}$. The set of regionally proximal pairs of
order $d$ is denoted by $\RP^{[d]}$ (or by $\RP^{[d]}(X,T)$ in case of
ambiguity), and is called {\em the regionally proximal relation of
order $d$}.
\end{de}

It is easy to see that $\RP^{[d]}$ is a closed and invariant
relation. Observe that
\begin{equation*}
    {\bf P}(X,T)\subseteq  \ldots \subseteq \RP^{[d+1]}\subseteq
    \RP^{[d]}\subseteq \ldots \subseteq \RP^{[2]}\subseteq \RP^{[1]}=\RP(X,T).
\end{equation*}
%\medskip

The following theorems proved in \cite{HKM} (for minimal distal systems) and
in \cite{SY} (for general minimal systems) tell us conditions under which
$(x,y)$ belongs to $\RP^{[d]}$ and the relation between $\RP^{[d]}$ and
$d$-step pro-nilsystems.

\begin{thm}\label{thm-1}\cite{SY}%[Theorem 3.4.]
Let $(X, T)$ be a minimal system and let $d\in \N$. Then
\begin{enumerate}
\item $(x,y)\in \RP^{[d]}$ if and only if $(x,y,\ldots,y)\in \Q^{[d+1]}$
if and only if $(x,y,\ldots,y) \in
\overline{\F^{[d+1]}}(x^{[d+1]})$.

\item $\RP^{[d]}$ is an equivalence relation.

\item $(X,T)$ is a system of order $d$ if and only if $\RP^{[d]}=\Delta_X$.
\end{enumerate}
\end{thm}

\begin{thm}\label{thm0}\cite{SY}\label{th3}
Let $\pi: (X,T)\rightarrow (Y,S)$ be a factor map between minimal systems
and let $d\in \N$. Then
\begin{enumerate}
  \item $\pi\times \pi (\RP^{[d]}(X,T))=\RP^{[d]}(Y,S)$.
  \item $(Y,T)$ is a system of order $d$ if and only if $\RP^{[d]}(X,T)\subset R_\pi$.
\end{enumerate}
In particular, the quotient of $(X,T)$ under $\RP^{[d]}(X,T)$ is the
maximal $d$-step pronilfactor of $X$ (i.e. the maximal factor of order
$d$).
\end{thm}

Let $Z_d=X/\RP^{[d]}(X,T)$ and $\pi_d: (X,T)\rightarrow (Z_d,T_d)$ be the factor map. $Z_0$ is the trivial system and $Z_1$ is the maximal equicontinuous factor $X_{eq}$.

\subsection{Some fundamental extensions}\

\medskip

Let $(X,T)$ and $(Y,S)$ be TDS and let $\pi: X \to Y$ be a factor map.
One says that $\pi$ is an {\it open} extension if it is open as a map;
and $\pi$ is a {\it semi-open} extension if the image of
every nonempty open set of $X$ has nonempty interior.
An important fact is that any factor map of minimal systems is semi-open.

\medskip

An extension $\pi$ is  {\it proximal} if
$\pi(x_1)=\pi(x_2)$ implies $(x_1,x_2) \in {\bf P}(X,T)$, and $\pi$ is   {\it distal} if $\pi(x_1)=\pi(x_2)$ and $x_1\neq x_2$ implies $(x_1,x_2) \not \in  {\bf P}(X,T)$. An extension $\pi$ is {\it almost one to one}  if there
exists a dense $G_\d$ set $X_0\subseteq X$ such that
$\pi^{-1}(\{\pi(x)\})=\{x\}$ for any $x\in X_0$. It is easy to see that any almost one to one extension between minimal systems is proximal.

%\item $\pi$ is an {\it equicontinuous or isometric} extension if for any $\ep >0$ there exists $\d>0$ such that $\pi(x_1)=\pi(x_2)$ and $d(x_1,x_2)<\d$ imply $d(T^n(x_1),T^n(x_2))<\ep$ for any $n\in \Z$;

%\item  $\pi$ is a {\em weakly mixing extension} if $(R_\pi, T\times T)$ as a subsystem of the product system $(X\times X, T\times T)$ is transitive.
%\end{enumerate}

%$\bullet$ $\pi$ is a {\it n-weak mixing} extension for some $n\geq 2$ if the system $(R^n_\pi,T^{(n)})$ is topologically transitive, where
%$$R^n_{\pi}=\{ (x_1,\ldots, x_n)\in X^n: \pi(x_1)=\ldots =\pi(x_n) \}.$$  If $\pi$ is $n$-weak mixing for any $n\ge 2$ then $\pi$ is said to be {\em totally weakly mixing}.

\medskip

An extension $\pi$ between minimal systems is called a {\em relatively incontractible (RIC) extension}\ if it is open and for every $n \ge 1$
the minimal points are dense in the relation
$$
R^n_\pi = \{(x_1,\dots,x_n) \in X^n : \pi(x_i)=\pi(x_j),\ \forall \ 1\le i
\le j \le n\}.
$$

A distal extension between minimal systems is RIC and that a RIC
extension is open. Every factor map between minimal systems can be
lifted to a RIC extension by proximal extensions (see \cite{EGS} or \cite[Chapter VI]{Vr}).

\begin{thm}\label{RIC}
Given a factor map $\pi:X\rightarrow Y$ between minimal systems
$(X,T)$ and $(Y,S)$ there exists a commutative diagram of factor
maps (called {\em RIC-diagram} or {\em EGS-diagram}\footnote{EGS stands for Ellis, Glasner and Shapiro \cite{EGS}.})

\[
\begin{CD}
X @<{\theta '}<< X'\\
@VV{\pi}V      @VV{\pi'}V\\
Y @<{\theta}<< Y'
\end{CD}
\]
such that
\begin{enumerate}
  \item[(a)] $\theta '$ and $\theta$ are proximal extensions;
  \item[(b)] $\pi '$ is a RIC extension;
  \item[(c)] $X '$ is the unique minimal set in $R_{\pi \theta
}=\{(x,y)\in X\times Y ': \pi(x)=\theta(y)\}$ and $\theta '$ and
$\pi '$ are the restrictions to $X '$ of the projections of $X\times
Y '$ onto $X$ and $Y '$ respectively.
\end{enumerate}
\end{thm}

\section{Topological characteristic factors along cubes}

Let $(X,T)$ be a minimal system and $d\in\N$. By Theorem \ref{th3}, $(Z_d,T_d)=(X/\RP^{[d]}(X),T_d)$ is the maximal $d$-step pro-nilfactor of $(X,T)$. For convenience, we still use the symbol $T$ as action on $Z_d$ instead of $T_d$, that is, $(Z_d,T)$ is the maximal $d$-step pro-nilfactor of $(X,T)$. Let $\pi_d: (X,T)\rightarrow (Z_d,T)$ be the factor map.

\medskip

In this section we will prove the main results of the paper. First we will show that modulo proximal extensions $\Q^{[d]}(X)$ is $\pi_{d-1}^{[d]}$-saturated. Then using this result we will prove that modulo proximal extensions the maximal $(d-1)$-step pro-nilfactor $(Z_{d-1}, T)$ is the topological cubic
characteristic factor of order $d$ of $(X,T)$.

\subsection{Parallelepiped $\Q^{[d]}$}\

\medskip

First we have the following useful lemma, which gives a condition when a point ${\bf x}\in X^{[d]}$ may be in $\Q^{[d]}(X)$.

\begin{lem}\label{lem1}
Let $(X,T)$ be a minimal system and $d\in\N$. Let $\pi:(X,T)\rightarrow (Z_{d-1},T)$ be the factor map to the maximal $(d-1)$-step pro-nilfactor. If points $x_1,x_2,\ldots, x_{2^d}\in X$ satisfy the following coditions:
\begin{enumerate}
 \item ${\bf x}=(x_1,x_2,\ldots,x_{2^d}) \in R_\pi^{2^d},$ that is, $\pi(x_1)=\pi(x_2)=\ldots =\pi(x_{2^d})$;
 \item $p_{1}:\overline{\mathcal{O}}({\bf x},T^{[d]})\rightarrow X$
is semi-open, where $p_{1}$ is the projection to the first coordinate,
\end{enumerate}
then $\{x_1,x_2,\ldots,x_{2^d}\}^{[d]}\subset \Q^{[d]}(X)$.
In particular, ${\bf x}=(x_1,x_2,\ldots,x_{2^d})\in \Q^{[d]}(X)$.
\end{lem}

\begin{proof}
We first prove the following claim.

\medskip
\noindent {\bf Claim:} {\em If $(x_1,\alpha_*)\in\{x_1,x_2,\ldots,x_{2^d}\}^{[d]}\cap \Q^{[d]}(X)$, then
\begin{equation*}
(x_i,\alpha_*)\in \Q^{[d]}(X) \ \text{ for any } i\in\{1,2,\ldots,2^d\}.
\end{equation*}
}

\begin{proof}[Proof of Claim]\renewcommand{\qedsymbol}{}
Fix an $i_0\in\{1,2,\ldots,2^d\}$, we will show that $(x_{i_0},\alpha_*)\in \Q^{[d]}(X)$.
Let $U_1,U_2,\ldots,U_{2^d}$ be neighborhoods of $x_1,x_2\ldots,x_{2^d}$ respectively.
Since $p_1$ is semi-open, one has that
\begin{equation*}
V_1={\rm int}\ p_1((U_1\times U_2\cdots\times U_{2^d})\cap \overline{\mathcal{O}}({\bf x},T^{[d]}))\neq\emptyset.
\end{equation*}
It is obvious that $V_1\subset U_1$. Set $V_2=U_2,\ldots,V_{2^d}=U_{2^d}$. Let $(x_1,\alpha_*)=(x_{s(\epsilon)})_{\epsilon\in\{0,1\}^d}$, where $s: \{0,1\}^d\rightarrow \{1,2,\ldots,2^d\}$ be a map with $s(\mathbf{0})=1$.

By the definition of $V_1$, it is easy to see that
\begin{equation*}
(V_1\times V_2\cdots\times V_{2^d})\cap{\mathcal{O}}({\bf x},T^{[d]})\neq\emptyset.
\end{equation*}
So  there exists $n_0\in\Z$ such that
 \begin{equation*}
 x_1\in T^{-n_0}V_1,x_2\in T^{-n_0}V_2,\ldots,x_{2^d}\in T^{-n_0}V_{2^d},
 \end{equation*}
then $(x_1,\alpha_*)=(x_{s(\epsilon)})_{\epsilon\in\{0,1\}^d}\in\prod\limits_{\epsilon\in\{0,1\}^d} T^{-n_0}V_{s(\epsilon)}$.

From the hypothesis, $(x_1,\alpha_*)\in \Q^{[d]}(X)$, then  there exists $\mathbf{n}\in\Z^d$ and  $x\in X$ such that
\begin{equation*}
    (T^{\mathbf{n}\cdot\epsilon}x)_{\epsilon\in\{0,1\}^d}\in\prod\limits_{\epsilon\in\{0,1\}^d} T^{-n_0}V_{s(\epsilon)}.
\end{equation*}
It follows that $\bigcap\limits_{\epsilon\in\{0,1\}^d} T^{-\mathbf{n}\cdot\epsilon}V_{s(\epsilon)}\neq\emptyset$. Set $W_1=\bigcap\limits_{\epsilon\in\{0,1\}^d} T^{-\mathbf{n}\cdot\epsilon}V_{s(\epsilon)}$
and $W_2=U_2,\ldots,W_{2^d}=U_{2^d}$. Note that $W_1\subset V_1$.

Since $W_1\subset V_1$, by the definition of $V_1$ one has that
\begin{equation*}
 (W_1\times W_2\cdots\times W_{2^d})\cap{\mathcal{O}}({\bf x},T^{[d]})\neq\emptyset.
 \end{equation*}
Thus there exists $n_1\in\Z$ such that
\begin{equation*}
  x_1\in T^{-n_1}W_1,x_2\in T^{-n_1}W_2,\ldots,x_{2^d}\in T^{-n_1}W_{2^d}.
 \end{equation*}
Since $(x_1,x_{i_0})\in R_\pi\subset \RP^{[d-1]}(X),$ by Theorem \ref{thm-1}
  \begin{equation*}
 (x_{i_0},(x_1^{[d]})_*)=(x_{i_0},x_1,\ldots,x_1)\in \Q^{[d]}(X).
 \end{equation*}
Note that $(x_{i_0},x_1,\ldots,x_1)\in T^{-n_1}W_{i_0}\times T^{-n_1}W_1\cdots\times T^{-n_1}W_1$, so there exist some $ \mathbf{m}\in\Z^d$ and $x^\prime\in X$ such that
    \begin{equation*}
 (T^{\mathbf{m}\cdot\epsilon}x^\prime)_{\epsilon\in\{0,1\}^d}\in T^{-n_1}W_{i_0}\times T^{-n_1}W_1\cdots\times T^{-n_1}W_1.
 \end{equation*}
   It follows that
    \begin{equation*}
W_{i_0}\cap\bigcap\limits_{\epsilon\in\{0,1\}^d\setminus\{\bf 0\}}T^{-\mathbf{m}\cdot\epsilon}W_1\neq\emptyset.
 \end{equation*}
 Since $W_1=\bigcap\limits_{\epsilon\in\{0,1\}^d} T^{-\mathbf{n}\cdot\epsilon}V_{s(\epsilon)}$, one has that
 \begin{equation*}
W_{i_0}\cap \bigcap_{\epsilon\in \{0,1\}^d\setminus \{\bf 0\}}T^{-{\bf m}\cdot \epsilon}\bigcap _{\eta\in \{0,1\}^d} T^{- {\bf n}\cdot \eta}V_{s(\eta)} \neq \emptyset.
 \end{equation*}
In particular, one has that
 \begin{equation*}
W_{i_0}\cap \bigcap_{\epsilon\in \{0,1\}^d\setminus \{\bf 0\}}T^{-({\bf m}+{\bf n})\cdot \epsilon}V_{s(\epsilon)}\neq \emptyset.
\end{equation*}
Since $W_i\subset V_i\subset U_i,  i\in\{1,2,\ldots 2^d\}$, it follows that
 \begin{equation*}
U_{i_0}\cap \bigcap_{\epsilon\in \{0,1\}^d\setminus \{\bf 0\}}T^{-({\bf m}+{\bf n})\cdot \epsilon}U_{s(\epsilon)}\neq \emptyset.
\end{equation*}
 Note that $U_i$ is arbitrary for each $i\in \{1,2,\ldots, 2^d\}$, by definition one has that $(x_{i_0},\alpha_*)\in \Q^{[d]}(X)$.
The proof of claim is completed.
\end{proof}

Now we begin to prove the lemma.
Let $\mathbf{y}\in\{x_1,x_2,\ldots,x_{2^d}\}^{[d]}$ and $l(\mathbf{y})$ denote the number of $x_1$'s appearing in $\mathbf{y}$. We prove the lemma by induction on $l({\bf y})$. If $l(\mathbf{y})=2^d,$ then $ \mathbf{y}=(x_1,x_1,\ldots,x_1)\in \Q^{[d]}(X).$

Assume that ${\bf y}\in \Q^{[d]}(X)$ whenever $l(\mathbf{y})= k\geq1$. We show that if $l(\mathbf{y})=k-1$ then ${\bf y}=(y_\epsilon)_{\epsilon\in \{0,1\}^d}\in \Q^{[d]}(X)$. Since $l({\bf y})=k-1<2^d$, there exists $\epsilon_0\in \{0,1\}^d$ such that $y_{\epsilon_0}\neq x_1$.

Let $I_0=\{i:(\epsilon_0)_i=0\}; I_1=\{i:(\epsilon_0)_i=1\}$ and define
$$
\phi:\{0,1\}^{[d]}\rightarrow\{0,1\}^{[d]}, (\phi(\epsilon))_i=\left\{
             \begin{array}{ll}
               \epsilon_i & i\in I_0 \hbox{;} \\
               1-\epsilon_i & i\in I_1 \hbox{.}
             \end{array}
           \right.$$
Then $\phi^*:X^{[d]}\rightarrow X^{[d]}:(\phi^*\mathbf{x})_\epsilon=\mathbf{x}_{\phi(\epsilon)}$ is an Euclidean permutation. Note that $\Q^{[d]}(X)$ is invariant under $\phi^*,(\phi^*)^{-1}$.

By the definition of $\phi$, $(\phi^*\mathbf{y})_{\bf 0}={y}_{\epsilon_0}\neq x_1$. Let ${\bf z}=(z_\epsilon)_{\epsilon\in \{0,1\}^d}$ with $ z_{\bf 0}=x_1$ and $z_\epsilon =(\phi^*\mathbf{y})_\epsilon$ for all $\epsilon\in \{0,1\}^d\setminus \{\bf 0\}$. Then $\mathbf{z}\in\{x_1,x_2,\ldots,x_{2^d}\}^{[d]}$ and $l(\mathbf{z})=k.$ By the inductive assumption $\mathbf{z}=(x_1,{\bf z}_*)\in Q^{[d]}(X)$. By claim,  $\phi^*\mathbf{y}\in \Q^{[d]}(X)$. Thus $\mathbf{y}=(\phi^*)^{-1}(\phi^*\mathbf{y})\in \Q^{[d]}(X)$. The proof is completed.
\end{proof}

By Lemma \ref{lem1}, one has the following corollary immediately.

\begin{cor}\label{cor-s1}
Let $(X,T)$ be a minimal system and $d\in \N$. Let $\pi:(X,T)\rightarrow (Z_{d-1},T)$ be the factor map to the maximal $(d-1)$-step pro-nilfactor. If ${\bf x}\in R_\pi^{2^d}$ is a $T^{[d]}$-minimal point, then ${\bf x}\in \Q^{[d]}(X)$.
\end{cor}

\begin{proof}
Since ${\bf x}\in R_\pi^{2^d}$ is a $T^{[d]}$-minimal point, $p_{1}:\overline{\O}({\bf x},T^{[d]})\rightarrow X$ is semi-open. The result follows from Lemma \ref{lem1}.
\end{proof}

Now we have that if the factor map to the maximal $(d-1)$-step pro-nilfactor is RIC, then $\Q^{[d]}(X)$ is $\pi^{[d]}$-saturated.

\begin{prop}\label{prop-RIC}
Let $(X,T)$ be a minimal system and $d\in \N$. Let $\pi:(X,T)\rightarrow (Z_{d-1},T)$ be the factor map to the maximal $(d-1)$-step pro-nilfactor. If $\pi$ is RIC, then
\begin{equation*}
\Q^{[d]}(X)=(\pi^{[d]})^{-1}\Q^{[d]}(Z_{d-1}).
\end{equation*}
\end{prop}

\begin{proof}
First it is obvious that $(\pi^{[d]})^{-1}\Q^{[d]}(Z_{d-1})\supset \Q^{[d]}(X)$. Now we show the other direction: $(\pi^{[d]})^{-1}\Q^{[d]}(Z_{d-1})\subset \Q^{[d]}(X)$.

If $\textbf{x}\in R_\pi^{2^d}$ is a $T^{[d]}$-minimal point, then by Corollary \ref{cor-s1}, $\textbf{x}\in \Q^{[d]}(X)$. Since $\pi$ is RIC, the set of $T^{[d]}$-minimal points is dense in $R_\pi^{2^d}$ and it follows that $R_\pi^{2^d}=(\pi^{[d]})^{-1}\Delta_{Z_{d-1}}\subset \Q^{[d]}(X)$.
Since $\pi$ is RIC, $\pi^{[d]}$ is open and $(\pi^{[d]})^{-1}$ is continuous. Thus
$$(\pi^{[d]})^{-1}\Q^{[d]}(Z_{d-1})=(\pi^{[d]})^{-1}\overline{\mathcal{G}^{[d]}\Delta_{Z_{d-1}}}=
\overline{\mathcal{G}^{[d]}(\pi^{[d]})^{-1}}\subset \Q^{[d]}(X).$$
The proof is completed.
\end{proof}

\medskip

We point out that we only use the fact that $R_\pi\subset \RP^{[d-1]}(X)$ in the proofs above. Since a distal extension is RIC, we have the following corollary.

\begin{cor}\label{cube-distal}
Let $(X,T)$ be a minimal distal system and $d\in \N$. Let $\pi:(X,T)\rightarrow (Z_{d-1},T)$ be the factor map to the maximal $(d-1)$-step pro-nilfactor. Then
\begin{equation*}
\Q^{[d]}(X)=(\pi^{[d]})^{-1}\Q^{[d]}(Z_{d-1}).
\end{equation*}
\end{cor}

\begin{prop}
$Z_{d-1}$ is the minimal factor such that Corollary \ref{cube-distal} holds.

\end{prop}

\begin{proof}
Assume $(Y,T)$ is a factor such that $\Q^{[d]}(X)=(\pi^{[d]})^{-1}\Q^{[d]}(Y)$ holds,
 we only need to show $R_\pi\subset \RP^{[d-1]}(X)$.
 To see this, assume $(x_1,x_2)\in R_\pi,$
 then $(\pi^{[d]})(x_1,x_2,\ldots,x_2)\in \Delta_Y\subset\Q^{[d]}(Y),$
 so $(x_1,x_2,\ldots,x_2)\in(\pi^{[d]})^{-1}\Q^{[d]}(Y)=\Q^{[d]}(X)$ and $ (x_1,x_2)\in\RP^{[d-1]}(X)$ by Theorem \ref{thm-1}.
\end{proof}
\medskip

Generally, we have that modulo proximal extensions $\Q^{[d]}(X)$ is $\pi^{[d]}$-saturated.

\begin{thm}\label{thm-main1}
Let $(X,T)$ be a minimal system and $d\in \N$. Let $\pi:(X,T)\rightarrow (Z_{d-1},T)$ be the factor map to the maximal $(d-1)$-step pro-nilfactor. Then there is a commutative diagram of homomorphisms of minimal flows
\begin{equation*}
  \xymatrix{
    X\ar[d]_{\pi}& X^\prime\ar[d]^{\pi^\prime}\ar[l]_{\theta^\prime} \\
    Z_{d-1}  &Y^\prime\ar[l]_{\theta}
    }
\end{equation*}
such that $\Q^{[d]}(X^\prime)=(\pi^{\prime[d]})^{-1}\Q^{[d]}(Y^\prime),$
where $\theta,\theta^\prime$ are proximal extensions.
\end{thm}

\begin{proof}
We only need to prove $R_{\pi^\prime}\subset \RP^{[d-1]}(X^\prime)$.
Let $(x_1^\prime,x_2^\prime)\in R_{\pi^\prime}$, then $\pi^\prime(x_1^\prime)=\pi^\prime(x_2^\prime)$ and
$\theta\pi^\prime(x_1^\prime)=\theta\pi^\prime(x_2^\prime)$, since the diagram is commutative, we have $\pi\theta^\prime(x_1^\prime)=\pi\theta^\prime(x_2^\prime)$, $(\theta^\prime(x_1^\prime),\theta^\prime(x_2^\prime))\in R_\pi=\RP^{[d-1]}(X).$

By theorem \ref{th3}, there is $(x_1^{\prime\prime},x_2^{\prime\prime})\in \RP^{[d-1]}(X^\prime),
(\theta^\prime (x_1^{\prime\prime}),\theta^\prime(x_2^{\prime\prime}))=(\theta^\prime(x_1^\prime),\theta^\prime(x_2^\prime))$.
Since $\theta^\prime$ is proximal, $(x_1^\prime,x_1^{\prime\prime}),(x_2^\prime,x_2^{\prime\prime})\in R_{\theta^\prime}\subset P(X^\prime)\subset \RP^{[d-1]}(X^\prime)$, since $\RP^{[d-1]}(X^\prime)$ is an equivalence relation, we have $(x_1^\prime,x_2^\prime)\in \RP^{[d-1]}(X^\prime)$, so $R_{\pi^\prime}\subset \RP^{[d-1]}(X^\prime).$

The result follows from Proposition \ref{prop-RIC}.
\end{proof}

\subsection{A counterexample}\
\medskip

Let $\pi:(X,T)\rightarrow (Z_{d-1},T)$ be the factor map. We use the following classical system to show that without additional conditions, $\Q^{[d]}(X)$ may not be $\pi^{[d]}$-saturated.

\begin{exam} Sturmian system.
\end{exam}
Let $\a$ be an irrational number in the interval $(0,1)$ and
$R_\a$ be the irrational rotation on the (complex) unit circle $\T$ generated
by $e^{2\pi i\a}$. Set
$$A_0=\left\{e^{2\pi i\theta}:0\leq \theta < (1-\a) \right\} \text{ and }
A_1=\left\{e^{2\pi i\theta}: (1-\a) \leq \theta < 1\right\}.$$

Consider $z \in \T$ and define $x \in \{0,1\}^{\Z}$ by: for all
$n\in \Z$, $x_{n}=i$ if and only if $R_{\a}^n(z) \in A_{i}$. Let  $X
\subset \{0,1\}^{\Z}$ be the orbit closure of $x$ under the shift
map $\sigma$ on $\{0,1\}^{\Z}$, i.e. for any $y \in \{0,1\}^{\Z}$,
$(\sigma(y))_n=y_{n+1}$. This system is called Sturmian system. It
is well known that $(X,\sigma)$ is a minimal almost one-to-one
extension of $(\T,R_\a)$. Moreover, it is an asymptotic extension.

Let $\pi: X \to \T$ be the former extension and consider
$(x_1,x_2)\in R_{\pi} \setminus \Delta_X$. Then $(x_1,x_2)$ is an
asymptotic pair and thus $(x_1,x_2)\in \RP^{[d]}$ for any integer
$d\geq 1$. It is showed in \cite[Example 4.8]{D-Y} that
$\{x_1,x_2\}^d\not \subset \Q^{[d]}(X).$ Hence
$$\Q^{[d]}(X)\neq (\pi^{[d]})^{-1}(\pi^{[d]}(\Q^{[d]}(X))).$$
That is, $\Q^{[d]}(X)$ is not $\pi^{[d]}$-saturated.

\subsection{Topological characteristic factors along cubes}\

\medskip

In this subsection we will use results developed above to show that up to proximal extensions the maximal $(d-1)$-step pro-nilfactor is the topological cubic
characteristic factor of order $d$. Before that, we use a different method to deal with distal systems.

\medskip

First we need some lemmas. By the proof of \cite[Theorem 3.1.]{SY} one can show the following lemma, which one can find another proof in \cite{G14}.

\begin{lem}\label{lem-sy-face}
Let $(X,T)$ be a system and $d \in\N$. If $\mathbf{x}\in X^{[d]}$ is an $\id \times T^{[d]}_*$-minimal point, then $\mathbf{x}$ is a $\mathcal{F}^{[d]}$-minimal point.
\end{lem}

We set $Q^{[d]}[x]=\{\mathbf{z}\in Q^{[d]}(X):z_{\bf 0}=x\}$.

\begin{lem}\cite[Proposition 5.2.]{SY}\label{pro1}
Let $(X,T)$ be a minimal system and $d\in\N.$ If $\mathbf{x}\in Q^{[d]}[x],$ then $x^{[d]}\in\overline{\mathcal{F}^{[d]}}(\mathbf{x})$. Especially, $(\overline{\F^{[d]}}(x^{[d]}), \F^{[d]})$ is the unique $\F^{[d]}$-minimal subset in $\Q^{[d]}[x]$.
\end{lem}

\begin{lem}\label{th4}
Let $(X,T)$ be a minimal system and $d\in \N$. Let $\pi:(X,T)\rightarrow (Z_{d-1},T)$ be the factor map to the maximal $(d-1)$-step pro-nilfactor. Assume that ${\bf x}\in R_\pi^{2^d}$ is an $\id\times T^{[d]}_*$-minimal point. Then  ${\bf x}\in \overline{\mathcal{F}^{[d]}}(x_{\bf 0}^{[d]})$.
\end{lem}

\begin{proof}
Assume that ${\bf x}\in R_\pi^{2^d}$ is an $\id\times T^{[d]}_*$-minimal point.
Let $x=x_{\bf 0}$. We show that ${\bf x}\in \overline{\mathcal{F}^{[d]}}(x^{[d]})$.
First we have the following claim, whose proof is given later.

\medskip
\noindent {\bf Claim:} {There is some $x'\in X$ such that $(x,x')\in R_\pi$ and
$$(x^{[d]}_*,x')=(x,x,\ldots,x,x')\in \overline{\F^{[d]}}({\bf x}).$$}

Since $R_\pi=\RP^{[d-1]}(X)$, $(x^{[d]}_*,x')\in \Q^{[d]}[x]$ by \cite[Lemma 6.2]{SY}.
Since ${\bf x}\in R_\pi^{2^d}$ is an $\id\times T^{[d]}_*$-minimal point, by Lemma \ref{lem-sy-face}, ${\bf x}$ is a $\mathcal{F}^{[d]}$-minimal point and so does $(x^{[d]}_*,x')$. By Proposition \ref{pro1}$(\overline{\F^{[d]}}(x^{[d]}), \F^{[d]})$ is the unique $\F^{[d]}$-minimal subset in $\Q^{[d]}[x]$, and one has that  $(x^{[d]}_*,x')\in\overline{\mathcal{F}^{[d]}}(x^{[d]})$. Since ${\bf x}$ is a
$\mathcal{F}^{[d]}$-minimal point and $(x^{[d]}_*,x')\in \overline{\F^{[d]}}({\bf x})$, one has that ${\bf x}\in \overline{\mathcal{F}^{[d]}}(x^{[d]})$.

\medskip

Now we give the proof of Claim. The idea of the proof of Claim is from the proof of \cite[Theorem 6.4]{SY}. We show the case $d=3$, and general case is similar.
Let
\begin{equation*}
     {\bf x}=(x_{000},x_{100},x_{010},x_{110},x_{001},
    x_{101},x_{011},x_{111}).
\end{equation*}
Then $x=x_{000}$.

By Proposition \ref{pro1},
there is some sequence $F^1_k\in \F^{[2]}$ such that
$$F_k^1(x_{000},x_{100},x_{010},x_{110})\to x^{[2]}=(x,x,x,x) , \ k\to \infty .$$
We may assume that
$$ F_k^1(x_{001},x_{101},x_{011},x_{111})\to (x_{001},x_{101}',x'_{011},x'_{111}),\ k\to \infty .$$
Since $F_k^1\times F_k^1\in \F^{[3]}$, one has that
$$(x,x,x,x,x_{001},x'_{101},x'_{011},x'_{111})\in \overline {\F^{[3]}}({\bf x}).$$

By Proposition \ref{pro1},
there is some sequence $F^2_k\in \F^{[2]}$ such that
$$F_k^2(x,x,x_{001},x'_{101})\to x^{[2]}=(x,x,x,x) , \ k\to \infty .$$
We may assume that
$$ F_k^2(x,x,x'_{011},x'_{111})\to (x,x,x''_{011},x''_{111}),\ k\to \infty .$$
Let $F_k^2=(F_k^{21}, F^{22}_k)$, where $F_k^{21}$ and $F^{22}_k$ act on $X^2$.
Then $(F_k^{21}, F_k^{21}, F^{22}_k, F^{22}_k)\in \F^{[3]}$, one has that
$$(x,x,x,x,x,x,x''_{011},x''_{111})\in \overline {\F^{[3]}}({\bf x}).$$

Again by Proposition \ref{pro1},
there is some sequence $F^3_k\in \F^{[2]}$ such that
$$F_k^3(x,x,x,x''_{011})\to x^{[2]}=(x,x,x,x) , \ k\to \infty .$$
We may assume that
$$ F_k^3(x,x,x,x''_{111})\to (x,x,x,x'),\ k\to \infty .$$
Let $F_k^3=(f_k^{1}, f_k^2,f_k^3,f^4_k)$.
Then $(f_k^{1},f_k^{1}, f_k^2, f_k^2,f_k^3,f_k^3,f^4_k,f^4_k)\in \F^{[3]}$, one has that
$$(x,x,x,x,x,x,x,x')\in \overline {\F^{[3]}}({\bf x}).$$
It is easy to check that $(x,x')\in R_\pi$. The proof is complete.
\end{proof}

By Lemma \ref{th4} one can deal with distal systems.

\begin{cor}\label{thm-distal}
Let $(X,T)$ be a minimal distal system and $d\in \N$. Let $\pi:(X,T)\rightarrow (Z_{d-1},T)$ be the factor map to the maximal $(d-1)$-step pro-nilfactor. Then
for each $x\in X$ with $y=\pi(x)$ one has that
\begin{equation}\label{s1}
\overline{\F^{[d]}}(x^{[d]})=\{x\}\times (\pi^{[d]}_*)^{-1}\overline{\F^{[d]}_*}(y^{[d]}_*).
\end{equation}
In particular, $Z_{d-1}$ is the topological cubic characteristic factor of order $d$.
\end{cor}

\begin{rem}
In the definition of the topological cubic characteristic factor of order $d$, we require (\ref{s1}) holds for a dense $G_\d$ set. But for distal systems, Corollary \ref{thm-distal} shows (\ref{s1}) holds for all $x\in X$.
\end{rem}

\medskip

By the method in \cite{AG95} or \cite[Section 4.]{G94}, one can prove the following result. We omit the proof here, and please refer to \cite{AG95} and \cite{G94} for more details about the methods.

\begin{lem}\label{lem-AG}
Let $(X,T)$ be a minimal system and $d\in \N$. There exists a dense $G_\delta$ set $X_0\subset X$ such that for each $x\in X_0$ one has that
\begin{equation*}
\Q^{[d]}[x]=\overline{\F^{[d]}}(x^{[d]}).
\end{equation*}
\end{lem}

\begin{prop}\label{pro2}
Let $(X,T)$ be a minimal system and $d\in \N$. Let $\pi:(X,T)\rightarrow (Z_{d-1},T)$ be the factor map to the maximal $(d-1)$-step pro-nilfactor. If $\pi$ is RIC, then $Z_{d-1}$ is the topological cubic characteristic factor of order $d$. That is, there exists a dense $G_\delta$ set $X_0\subset X$ such that
for each $x\in X_0$ with $y=\pi(x)$ one has that
\begin{equation*}
\overline{\F^{[d]}}(x^{[d]})=\{x\}\times (\pi^{[d]}_*)^{-1}\overline{\F^{[d]}_*}(y^{[d]}_*).
\end{equation*}
\end{prop}

\begin{proof}
It follows from Proposition \ref{prop-RIC} and Lemma \ref{lem-AG}.
\end{proof}

\begin{thm}\label{thm-main2}
Let $(X,T)$ be a minimal system and $d\in \N$. Let $\pi:(X,T)\rightarrow (Z_{d-1},T)$ be the factor map to the maximal $(d-1)$-step pro-nilfactor. Then there is a commutative diagram of homomorphisms of minimal flows
\begin{equation*}
  \xymatrix{
    X\ar[d]_{\pi}& X^\prime\ar[d]^{\pi^\prime}\ar[l]_{\theta^\prime} \\
    Z_{d-1}  &Y^\prime\ar[l]_{\theta}
    }
\end{equation*}
such that  $(Y',T)$ is the topological cubic
characteristic factor of order $d$ of $(X',T)$,
where $\theta,\theta^\prime$ are proximal extensions.
\end{thm}

\begin{proof}
It follows from Theorem \ref{thm-main1} and Proposition \ref{pro2}.
\end{proof}

For weakly mixing systems, we have the following theorem, which was first proved in \cite{SY}.

\begin{cor}\cite{SY}
Let $(X,T)$ be a minimal weakly mixing system and $d\in \N$. Then
\begin{enumerate}
  \item $(\Q^{[d]}, \G^{[d]})$ is minimal and $\Q^{[d]}=X^{[d]}$.
  \item For all $x\in X$, $(\overline{\F^{[d]}}(x^{[d]}), \F^{[d]})$
  is minimal and $$\overline{\F^{[d]}}(x^{[d]})=\{x\}\times X^{[d]}_*=\{x\}\times
  X^{2^d-1}.$$
\end{enumerate}
\end{cor}

%{\red
%\begin{prop}
%Let $(X,T)$ be a minimal system and $d\in \N$. Let $\pi:(X,T)\rightarrow (Z_{d-1},T)$ be the factor map to the maximal $(d-1)$-step pro-nilfactor. If $\pi$ is almost one to one, then $Z_{d-1}$ is the topological cubic characteristic factor of order $d$.
%\end{prop}

%\begin{proof}

%\end{proof}

%}

%%%%%%%%%%%%%%%%%%%%%%%%%%%%%%%%%%%%%%%%%%%%%%%%%%%%%%%%

\end{document}